\documentclass[11pt]{article}

\usepackage[a4paper,margin=30mm]{geometry}
\usepackage[T1]{fontenc}
\usepackage{lmodern}
\usepackage{microtype}
\usepackage{amsmath,amssymb,amsthm,mathtools}
\usepackage{booktabs,tabularx,array}
\usepackage{xcolor}
\usepackage{enumitem}
\usepackage{hyperref}
\usepackage[nameinlink,capitalize,noabbrev]{cleveref}

\definecolor{linkblue}{RGB}{35,75,125}
\hypersetup{
  colorlinks=true,
  linkcolor=linkblue,
  citecolor=linkblue,
  urlcolor=linkblue,
  pdftitle={Some Integrability Properties of m-Subharmonic Functions}
}

\setlist{itemsep=0.2em,topsep=0.4em}

\newtheorem{theorem}{Theorem}[section]
\newtheorem{proposition}[theorem]{Proposition}
\newtheorem{lemma}[theorem]{Lemma}
\newtheorem{corollary}[theorem]{Corollary}
\theoremstyle{definition}
\newtheorem{definition}[theorem]{Definition}

\theoremstyle{remark}
\newtheorem{remark}[theorem]{Remark}

\newcommand{\C}{\mathbb C}
\newcommand{\SH}{\mathrm{SH}}
\newcommand{\mCap}{\operatorname{Cap}}
\newcommand{\Vol}{\operatorname{Vol}}
\newcommand{\loc}{\mathrm{loc}}
\newcommand{\ddc}{dd^c}
\newcommand{\pc}{p_{\mathrm c}}
\newcommand{\Cm}{(\mathrm C_m)}
\newcommand{\Cmd}{(\mathrm C_{m,\delta})}

\title{\bfseries Some Integrability Properties of \(m\)-Subharmonic Functions}
\author{Genglong Lin\\[0.35em]
\small Beijing Institute of Mathematical Sciences and Applications,
Beijing 101408, China\\
\small\href{mailto:lingenglong@bimsa.cn}{\texttt{lingenglong@bimsa.cn}}}
\date{July 2026}

\begin{document}
\maketitle
\vspace{-2.2em}

\begin{abstract}
Let \(1\le m<n\) and let \(u\) be an \(m\)-subharmonic function on a
domain in \(\C^n\).  We study local exponential and polynomial
integrability, with particular attention to the sharp polynomial exponent
predicted by B{\l}ocki's conjecture.  Explicit radial examples show that
direct analogues of the Guan--Zhou strong openness theorem and Skoda's
integrability criterion formulated in terms of the \(m\)-Lelong number fail
when \(m<n\).  We classify a family of radial power-logarithmic
singularities and determine the exact \(L^p\)-integrability range for each
member, including endpoint behavior.

We resolve two problems posed by Benali--Ghiloufi.  The normalized limit of
the ball maximum always equals the \(m\)-Lelong number; this follows by
combining their spherical-mean formula with the strong uniqueness theorem
for tangents.  The pointwise integrability exponent is lower semicontinuous
in the base point.  However, even when restricted to \(\SH_m\), it is not
lower semicontinuous with respect to the \(L^1_{\loc}\) topology.  We also
disprove their polynomial openness conjecture using an explicit
power-logarithmic endpoint example.

Finally, we introduce a scale of local Hessian-capacity conditions, denoted
by \(\Cmd\).  The volume-capacity inequality and the layer-cake formula yield
\[
 u\in L^s_{\loc}\quad\text{for every}\quad
 s<\frac{(m+\delta)n}{n-m}.
\]
The critical condition \(\mathrm C_{m,0}=\mathrm C_m\) holds for negative
functions of finite total Hessian mass with relatively compact deep sublevel
sets, and for radial functions.  More generally, functions in the energy
class \(\mathcal E_{p,m}\) satisfy \(\mathrm C_{m,p}\),
recovering the full {\AA}hag--Czy{\.z} Sobolev exponent.  These results
provide partial progress toward B{\l}ocki's conjecture, which has remained
open for more than two decades.
\end{abstract}

\noindent\textbf{Keywords.}
\(m\)-subharmonic function; complex Hessian equation; integrability
exponent; Hessian capacity; Lelong number; strong openness.

\medskip
\noindent\textbf{MSC 2020.} 32U05, 32U25, 32W20, 35J60.

\section{Introduction}

The cone of \(m\)-subharmonic functions interpolates between ordinary
subharmonic functions and plurisubharmonic functions:
\[
  \operatorname{PSH}=\SH_n\subset\SH_{n-1}\subset\cdots
  \subset\SH_1=\operatorname{SH}.
\]
The two endpoints have markedly different singularity theories.  A
plurisubharmonic germ has nontrivial local exponential integrability, while
an ordinary subharmonic function in real dimension \(2n\) may have a
Newtonian pole.  For \(1\le m<n\), the fundamental complex Hessian pole is
\[
  G_m(z)=-|z|^{-2(n-m)/m}.
\]
It belongs locally to \(L^p\) exactly when
\[
  p<\pc:=\frac{nm}{n-m}.
\]
This model led B{\l}ocki to conjecture that every local
\(m\)-subharmonic function belongs to \(L^p_{\loc}\) for every
\(p<\pc\) \cite{Blocki2005}.  The endpoint is necessarily excluded:
\(G_m\notin L^{\pc}_{\loc}\).

There are three different integrability questions that should not be
conflated.  First, one can ask whether the local integrability of \(e^{-u}\)
implies that of \(e^{-pu}\) for some \(p>1\).  Second, one can ask
whether the \(m\)-Lelong number controls exponential integrability, as in
Skoda's theorem for plurisubharmonic functions.  Third, one can study
positive powers \(|u|^p\).  We give explicit negative answers to the first
two questions for every \(m<n\), whereas the third includes B{\l}ocki's
sharp conjecture, which remains largely open.

\paragraph{Contrast with the plurisubharmonic case.}
When \(m=n\), the model singularity is logarithmic.  Every
plurisubharmonic germ that is not identically \(-\infty\) has some local
exponential integrability and therefore belongs locally to \(L^p\) for
every finite \(p\).  Skoda's theorem relates the exponential threshold to
the ordinary Lelong number, while the Guan--Zhou theorem shows that the
range of integrable exponential weights cannot have a finite closed
endpoint.  Thus polynomial integrability is automatic in the
plurisubharmonic setting.

For \(m<n\), by contrast, \(m\)-subharmonic functions may have algebraic
poles of Hessian type.  Their polynomial integrability can stop at a
finite critical exponent, and logarithmic corrections decide whether the
endpoint is included.  Moreover, a vanishing \(m\)-Lelong number does not
ensure any exponential integrability, and the direct analogue of strong
openness fails.  Consequently, the \(m<n\) problem cannot be treated by
formally substituting \(m\)-subharmonic functions and their Lelong numbers
into the classical theorems.  Instead, one needs estimates for polynomial
sublevel tails; in this paper, Hessian capacity is the controlling quantity.

The main positive tool in this paper is the capacity-decay scale
\[
 \mCap_m(K\cap\{u<A-t\},D)
 \le C t^{-(m+\delta)}.
 \tag{\(\mathrm C_{m,\delta}\)}\label{eq:capacity-condition-intro}
\]
Combined with the Dinew--Ko{\l}odziej volume-capacity inequality
\cite{DinewKolodziej2014}, this condition yields
\(u\in L^s_{\loc}\) for \(s<(m+\delta)n/(n-m)\).  The critical condition
\(\mathrm C_{m,0}=\mathrm C_m\) gives B{\l}ocki's conjectured strict range
and covers negative functions of finite total Hessian mass with relatively
compact deep sublevel sets, the local Cegrell class, and radial germs.  The
stronger condition \(\mathrm C_{m,p}\) holds for
\(\mathcal E_{p,m}\) and recovers the improved {\AA}hag--Czy{\.z} exponent.
The assertion for radial germs follows from the mean-value formula of
Benali--Ghiloufi \cite{BenaliGhiloufi2018}.

A second set of results addresses two questions left open by
Benali--Ghiloufi.  They proved that the normalized limit of the ball maximum
\(M(u,a,r)\) exists and is bounded above by \(\nu_u(a)\).  We prove the
reverse inequality using the strong uniqueness theorem for tangents when
\(m<n\), thereby answering their Problem~2.  We also resolve all parts of
their Problem~3: lower semicontinuity with respect to the base point holds,
whereas lower semicontinuity with respect to the function fails in the
\(L^1_{\loc}\) topology, and the polynomial integrability interval need not
be open.

The paper is organized as follows.  \Cref{sec:prelim} establishes the radial
Hessian test.  \Cref{sec:models} classifies the \(m\)-subharmonic
power-logarithmic models and determines their exact \(L^p\)-ranges.
\Cref{sec:exponential} gives counterexamples to direct analogues of the
strong openness theorem and Skoda's integrability criterion.
\Cref{sec:polynomial} treats
polynomial exponents and a failure of endpoint openness.
\Cref{sec:blocki} distinguishes the false endpoint statement from
B{\l}ocki's original conjecture.  \Cref{sec:capacity} proves the capacity criterion
and incorporates the {\AA}hag--Czy{\.z} energy classes.
\Cref{sec:bg-problems} answers Problems~2 and~3 of Benali--Ghiloufi, and
\Cref{sec:radial-bg} incorporates their radial mean-value mechanism into
\(\Cm\).  \Cref{sec:future} discusses the main open directions suggested by
these results.

\section{Preliminaries and the radial Hessian test}
\label{sec:prelim}

Let \(\Omega\subset\C^n\) be a domain and put
\(\beta=\ddc |z|^2\).  A locally integrable, upper semicontinuous function
\(u\) is \(m\)-subharmonic if it is subharmonic and
\[
  \ddc u\wedge\alpha_1\wedge\cdots\wedge\alpha_{m-1}
  \wedge\beta^{n-m}\ge0
\]
for all constant-coefficient \(m\)-positive \((1,1)\)-forms
\(\alpha_j\).  For \(C^2\) functions, this is equivalent to requiring the
eigenvalue vector of the complex Hessian to lie in the G{\aa}rding cone
\(\Gamma_m\); equivalently, its first \(m\) elementary symmetric functions
must be nonnegative \cite[Proposition 3.1]{Blocki2005}.

\begin{lemma}
\label{lem:radial-test}
Let \(u(z)=f(t)\), where \(t=|z|^2\) and \(f\in C^2(0,R^2)\).  Put
\[
  A=f'(t),\qquad B=f'(t)+tf''(t).
\]
The complex Hessian has tangential eigenvalue \(A\), with multiplicity
\(n-1\), and radial eigenvalue \(B\).  Thus
\[
 \sigma_k=\binom{n-1}{k}A^k+
 \binom{n-1}{k-1}A^{k-1}B,\qquad 1\le k\le n.
\]
If \(A>0\), then \(u\) is \(m\)-subharmonic on the punctured ball if and
only if
\[
  \frac BA\ge-\frac{n-m}{m}.
\]
\end{lemma}

\begin{proof}
Direct differentiation gives
\[
 u_{j\bar k}=f'(t)\delta_{jk}+f''(t)\bar z_j z_k.
\]
At a fixed nonzero point, a unitary change of coordinates allows us to take
\(z=(\sqrt t,0,\ldots,0)\).  The Hessian matrix is then diagonal with entries
\(B,A,\ldots,A\), which gives the formula for \(\sigma_k\).  If \(A>0\),
then, for each \(k\), the inequality \(\sigma_k\ge0\) is equivalent to
\[
 \frac BA\ge-\frac{\binom{n-1}{k}}{\binom{n-1}{k-1}}
 =-\frac{n-k}{k}.
\]
Among \(1\le k\le m\), the strongest condition corresponds to \(k=m\).
\end{proof}

\begin{lemma}
\label{lem:truncation}
Suppose that
\[
 u\in\SH_m(B(0,R)\setminus\{0\}),\qquad
 \lim_{z\to0}u(z)=-\infty,
\]
and that \(u\not\equiv-\infty\).  Defining \(u(0)=-\infty\), the function
\(u\) belongs to \(\SH_m(B(0,R))\).
\end{lemma}

\begin{proof}
For \(j\ge1\), set
\[
 u_j=\max\{u,-j\}
 \quad\text{on }B(0,R)\setminus\{0\}.
\]
The class \(\SH_m\) is closed under finite maxima
\cite[Proposition 3.1(v)]{Blocki2005}; hence \(u_j\) is
\(m\)-subharmonic on the punctured ball.  Since \(u(z)\to-\infty\), there is
\(r_j>0\) such that \(u<-j\) on
\(0<|z|<r_j\).  Thus \(u_j\equiv-j\) there.  After setting
\(u_j(0)=-j\), the extension is constant on \(B(0,r_j)\) and hence is
\(m\)-subharmonic at the origin.

Moreover,
\[
 u_{j+1}\le u_j,\qquad u_j\downarrow u.
\]
A decreasing sequence of \(m\)-subharmonic functions has an
\(m\)-subharmonic limit unless the limit is identically \(-\infty\)
\cite[Proposition 1(6)]{BenaliGhiloufi2018}.  Indeed, because the limit is
not identically \(-\infty\), the standard convergence theorem for
subharmonic functions gives \(u_j\to u\) in \(L^1_{\loc}\).  Hence
\(\ddc u_j\to\ddc u\) in the sense of distributions.  Wedging with fixed
smooth \(m\)-positive forms is continuous under distributional convergence,
and the cone of positive currents is closed.  Thus the defining inequalities
pass to the limit.
\end{proof}

\section{Power-logarithmic singularities}
\label{sec:models}

For \(\alpha,\gamma\in\mathbb R\) and \(0<|z|<1\), set
\[
  c_m=\frac{n-m}{m},\qquad L=-\log |z|^2,
\]
and consider
\[
 u_{\alpha,\gamma}(z)
 =-|z|^{-2\alpha}L^{-\gamma}.
\]

\begin{theorem}
\label{thm:radial-classification}
The following assertions hold on a sufficiently small punctured ball.
\begin{enumerate}[label=\textup{(\roman*)}]
\item If \(\alpha>0\), then \(u_{\alpha,\gamma}\) is
\(m\)-subharmonic if and only if either \(\alpha<c_m\), or
\(\alpha=c_m\) and \(\gamma\ge0\).
\item If \(\alpha=0\), then \(u_{\alpha,\gamma}\) is \(m\)-subharmonic
if and only if \(\gamma\le0\).
\item If \(\alpha<0\), then \(u_{\alpha,\gamma}\) is not
\(m\)-subharmonic near zero.
\end{enumerate}
Every admissible model with a singularity at zero extends across zero as an
\(m\)-subharmonic function.
\end{theorem}

\begin{proof}
Writing \(t=|z|^2\), one computes
\[
 A=t^{-\alpha-1}L^{-\gamma}
 \left(\alpha-\frac\gamma L\right)
\]
and, whenever \(A\ne0\),
\[
 \frac BA=-\alpha+\frac\gamma L
 -\frac{\gamma}{L^2(\alpha-\gamma/L)}.
\]
Assume first that \(\alpha>0\).  Then \(A>0\) for large \(L\) and
\(B/A\to-\alpha\).  If \(\alpha<c_m\), the inequality
\(B/A\ge-c_m\) is eventually strict; if \(\alpha>c_m\), it eventually
fails.  At \(\alpha=c_m\), the sign of the difference between the displayed
ratio and \(-c_m\) is, for large \(L\), the sign of
\[
 \gamma(Lc_m-\gamma-1).
\]
It is nonnegative for all sufficiently large \(L\) exactly when
\(\gamma\ge0\).

Next let \(\alpha=0\).  If \(\gamma\ne0\), direct differentiation gives
\[
 A=-\gamma t^{-1}L^{-\gamma-1},\qquad
 B=-\gamma(\gamma+1)t^{-1}L^{-\gamma-2},
 \qquad
 \frac BA=\frac{\gamma+1}{L}.
\]
For \(\gamma<0\), one has \(A>0\) and \(B/A\to0>-c_m\), so
\Cref{lem:radial-test} proves \(m\)-subharmonicity.  For \(\gamma=0\),
the function is the constant \(-1\).  For \(\gamma>0\), \(A<0\), and
the trace of the complex Hessian is
\[
 \sigma_1=(n-1)A+B
 =A\left(n-1+\frac{\gamma+1}{L}\right)<0
\]
for large \(L\).  The function is then not even subharmonic.

Finally, if \(\alpha<0\), then \(A<0\) for large \(L\), while
\[
 \frac BA\longrightarrow-\alpha>0.
\]
Consequently
\[
 \sigma_1=A\left(n-1+\frac BA\right)<0
\]
near zero.  Again the function is not subharmonic.  Every admissible
singular model tends to \(-\infty\) at zero, so
\Cref{lem:truncation} gives the claimed extension.
\end{proof}

\begin{proposition}
\label{prop:radial-lp}
If \(\alpha>0\) and \(p>0\), then
\[
 u_{\alpha,\gamma}\in L^p_{\loc}
 \quad\Longleftrightarrow\quad
 \begin{cases}
 p<n/\alpha,\quad\text{or}\\
 p=n/\alpha\ \text{and}\ \gamma p>1.
 \end{cases}
\]
If \(\alpha\le0\), then \(u_{\alpha,\gamma}\in L^p_{\loc}\) for every
\(p>0\).
\end{proposition}

\begin{proof}
Polar integration, followed by \(t=r^2\), gives
\[
 \int_{|z|<\varepsilon}|u_{\alpha,\gamma}|^p\,dV
 \asymp
 \int_0^{\varepsilon^2}
 t^{n-1-\alpha p}(-\log t)^{-\gamma p}\,dt.
 \tag{3.1}\label{eq:power-log-integral}
\]
Suppose first that \(\alpha>0\).  The integral in
\eqref{eq:power-log-integral} converges when \(n-1-\alpha p>-1\),
equivalently when \(p<n/\alpha\), and diverges when \(p>n/\alpha\).  At the
critical value
\(p=n/\alpha\), \eqref{eq:power-log-integral} becomes
\[
 \int_0^{\varepsilon^2}
 \frac{(-\log t)^{-\gamma p}}{t}\,dt.
\]
With \(L=-\log t\), this is
\[
 \int_{-\log\varepsilon^2}^{\infty}L^{-\gamma p}\,dL,
\]
which converges exactly when \(\gamma p>1\).  At
\(\gamma=\alpha/n\), it is the logarithmically divergent integral
\(\int^\infty dL/L\).

If \(\alpha=0\), the same substitution gives
\[
 \int_0^{\varepsilon^2}t^{n-1}(-\log t)^{-\gamma p}\,dt
 =\int_{-\log\varepsilon^2}^{\infty}
 e^{-nL}L^{-\gamma p}\,dL<\infty.
\]
The exponential factor dominates every power of \(L\).  In particular,
the admissible unbounded models with \(\gamma<0\) have only logarithmic
growth, and all their positive powers are locally integrable.  If
\(\alpha<0\),
write \(\alpha=-a\), \(a>0\).  Then
\eqref{eq:power-log-integral} becomes
\[
 \int_{-\log\varepsilon^2}^{\infty}
 e^{-(n+ap)L}L^{-\gamma p}\,dL<\infty.
\]
Alternatively,
\(|u_{\alpha,\gamma}(z)|
=|z|^{2a}(-\log|z|^2)^{-\gamma}\to0\), so the model is locally bounded.
\end{proof}

The boundary model \(u_{c_m,0}=G_m\) therefore has the exact range
\[
 G_m\in L^p_{\loc}\quad\Longleftrightarrow\quad
 p<\pc,\qquad \pc=\frac{n}{c_m}=\frac{nm}{n-m}.
\]
Moreover, \(G_m\in L^{\pc,\infty}_{\loc}\): for every sufficiently small
\(\varepsilon>0\) and all large \(s\),
\[
 \Vol\bigl(B(0,\varepsilon)\cap\{|G_m|>s\}\bigr)\asymp s^{-\pc}.
\]

\section{Failure of direct exponential analogues}
\label{sec:exponential}

For an \(m\)-subharmonic germ \(u\) at the origin, define the
exponential-integrability ideals
\[
 \mathcal I_m(u)_0
 =\{F\in\mathcal O_{\C^n,0}:|F|^2e^{-u}\in L^1_{\loc}\}
\]
and
\[
\mathcal I_{m,+}(u)_0
=\bigcup_{\varepsilon>0}\mathcal I_m((1+\varepsilon)u)_0
=\bigcup_{p>1}\mathcal I_m(pu)_0.
\]
When \(m=n\), the strong openness theorem of Guan--Zhou asserts that
\(\mathcal I_n(u)_0=\mathcal I_{n,+}(u)_0\) \cite{GuanZhou2015}; the special
case \(F\equiv1\) was proved earlier by Berndtsson
\cite{Berndtsson2013}.  The corresponding statement fails for every
\(m<n\).

\begin{theorem}
\label{thm:strong-open-failure}
Let \(n\ge2\) and \(1\le m<n\).  On \(|z|<e^{-1}\), set
\[
 u(z)=n\log|z|^2+2\log(-\log|z|^2),\qquad u(0)=-\infty.
\]
Then \(u\in\SH_m\), and for every \(p>0\),
\[
 e^{-pu}\in L^1_{\loc}\quad\Longleftrightarrow\quad p\le1.
\]
Consequently \(1\in\mathcal I_m(u)_0\) but
\(1\notin\mathcal I_{m,+}(u)_0\).
\end{theorem}

\begin{proof}
Put \(q=\log|z|^2\), \(T=-q\), and
\(h(q)=nq+2\log(-q)\).  The tangential and radial eigenvalues of the complex
Hessian are
\[
 a=\frac{n-2/T}{|z|^2}>0,\qquad
 b=-\frac{2}{T^2|z|^2}.
\]
For every \(1\le k\le m\),
\[
 \sigma_k=\binom{n-1}{k-1}a^{k-1}
 \left(\frac{n-k}{k}a+b\right)
\]
is positive on the chosen ball.  Since \(u(z)\to-\infty\) as
\(z\to0\), \Cref{lem:truncation} shows that the extension defined by
\(u(0)=-\infty\) remains \(m\)-subharmonic.  Moreover,
\(u=-nT+2\log T\), and a change of variables gives
\[
 \int_{|z|<\varepsilon}e^{-pu}\,dV
 \asymp\int_{-\log\varepsilon^2}^{\infty}
 e^{n(p-1)T}T^{-2p}\,dT.
\]
This integral converges if and only if \(p\le1\).  Since the radial
eigenvalue is negative, the example is not plurisubharmonic and hence does
not contradict the
Guan--Zhou theorem.
\end{proof}

With
\[
 \phi_m(r)=-\frac{1}
 {(n/m-1)r^{2(n/m-1)}}
\]
the Benali--Ghiloufi normalization \cite{BenaliGhiloufi2018} of the
\(m\)-Lelong number in terms of the spherical mean \(\lambda(u,a,r)\) is
\[
 \nu_u(a)=2\lim_{r\downarrow0}
 \frac{\lambda(u,a,r)}{\phi_m(r)}.
 \tag{4.1}\label{eq:lelong-spherical}
\]

\begin{theorem}
\label{thm:skoda-failure}
For every \(1\le m<n\), there exists an \(m\)-subharmonic germ \(u\) such
that \(\nu_u(0)=0\), yet \(e^{-\gamma u}\notin L^1_{\loc}\) for every
\(\gamma>0\).
\end{theorem}

\begin{proof}
Choose \(0<\alpha<c_m\) and take
\(u(z)=-|z|^{-2\alpha}\).  By
\Cref{thm:radial-classification}, this is \(m\)-subharmonic across zero.
Formula~\eqref{eq:lelong-spherical} gives
\[
 \frac{\lambda(u,0,r)}{\phi_m(r)}
 =c_m r^{2(c_m-\alpha)}\longrightarrow0.
\]
On the other hand,
\[
 \int_{|z|<\varepsilon}e^{-\gamma u}\,dV
 \asymp\int_0^\varepsilon
 e^{\gamma r^{-2\alpha}}r^{2n-1}\,dr=\infty.
\]
Since the radial Hessian eigenvalue is negative, this example cannot occur
in the plurisubharmonic case.
\end{proof}

\begin{remark}
\label{rem:finite-energy-range}
In \cite{Lin2025}, we studied the finite-energy range of the \(m\)-Hessian
operator on compact K\"ahler manifolds.  The corresponding Monge--Amp\`ere
argument relies on Skoda's exponential-integrability theorem, for which
there is no direct analogue when \(m<n\).  The proof in \cite{Lin2025}
therefore uses capacity estimates instead.  \Cref{thm:skoda-failure} shows
that this difference is structural: even the condition \(\nu_u(0)=0\) does
not guarantee \(e^{-\gamma u}\in L^1_{\loc}\) for any \(\gamma>0\).  Hence no
criterion based only on the \(m\)-Lelong number can provide the missing
exponential-integrability estimate in general.  This motivates the local
capacity scale introduced in \Cref{sec:capacity}.
\end{remark}

\section{Polynomial integrability exponents}
\label{sec:polynomial}

For \(u\in L^1_{\loc}(\Omega)\) and a compact set \(K\Subset\Omega\), define
\[
 \iota_K(u)=\sup\{p>0:|u|^p\in L^1(U)
 \text{ for some open }U\text{ with }K\Subset U\Subset\Omega\}.
\]
Benali--Ghiloufi showed \cite{BenaliGhiloufi2018} that \(\iota_K(u)\) is the
supremum of all powers \(\alpha>0\) for which the volumes of the deep
sublevel sets are \(O(t^{-\alpha})\).  They also proved that
\[
 \iota_x(u)\ge\frac{n}{n-m},\qquad
 \nu_u(x)>0\ \Longrightarrow\
 \iota_x(u)\le\frac{nm}{n-m}.
 \tag{5.1}\label{eq:bg-exponent-bounds}
\]
The first inequality is a universal lower bound, whereas the second is an
upper bound under the hypothesis \(\nu_u(x)>0\).  In particular, the latter
does not establish B{\l}ocki's conjectural lower bound.

\begin{proposition}
\label{prop:point-lsc}
Let \(u\in L^1_{\loc}(\Omega)\).  Then the map
\[
 a\longmapsto\iota_a(u)
\]
is lower semicontinuous on \(\Omega\).
\end{proposition}

\begin{proof}
Fix \(a\in\Omega\) and \(c<\iota_a(u)\).  Choose \(q\) such that
\(\max\{c,0\}<q<\iota_a(u)\).  By the definition of \(\iota_a(u)\) and the monotonicity
of local \(L^p\)-integrability in \(p\), the function \(|u|^q\) is integrable
on a neighborhood \(U\) of \(a\).  Choose a neighborhood \(V\) of \(a\)
with \(V\Subset U\).  Then \(\iota_b(u)\ge q>c\) for every \(b\in V\).
Thus \(\{b\in\Omega:\iota_b(u)>c\}\) is open, which proves the assertion.
\end{proof}

In contrast, the admissible set of exponents need not be open at its
endpoint.

\begin{proposition}
\label{prop:closed-endpoint}
Let \(0<\alpha<c_m\) and \(\gamma>\alpha/n\).  On a sufficiently small ball
centered at the origin, define
\[
 \varphi_{\alpha,\gamma}(z)
 =-|z|^{-2\alpha}(-\log|z|^2)^{-\gamma},\qquad z\ne0,
\]
and set \(\varphi_{\alpha,\gamma}(0)=-\infty\).  Then
\(\varphi_{\alpha,\gamma}\) is \(m\)-subharmonic and
satisfies
\[
 \{p>0:|\varphi_{\alpha,\gamma}|^p
 \text{ is integrable in a neighborhood of }0\}
 =\left(0,\frac n\alpha\right].
\]
\end{proposition}

\begin{proof}
After shrinking the ball if necessary,
\Cref{thm:radial-classification} gives the \(m\)-subharmonicity.
\Cref{prop:radial-lp} gives convergence for \(p<n/\alpha\) and divergence
for \(p>n/\alpha\).  At \(p=n/\alpha\), the integral converges precisely
when \(\gamma n/\alpha>1\), which follows from \(\gamma>\alpha/n\).
\end{proof}

\begin{table}[ht]
\centering
\small
\caption{Selected results on polynomial integrability.  Here
\(\pc=nm/(n-m)\).}
\label{tab:results}
\begin{tabularx}{\textwidth}{@{}>{\raggedright\arraybackslash}p{0.26\textwidth}
 >{\raggedright\arraybackslash}p{0.20\textwidth}
 >{\raggedright\arraybackslash}X@{}}
\toprule
Setting & Exponent range & Additional hypothesis or status\\
\midrule
Arbitrary local \(m\)-subharmonic functions
& \(p<n/(n-m)\)
& B{\l}ocki's universal bound \cite{Blocki2005}\\
Arbitrary local \(m\)-subharmonic functions
& \(p<\pc\)
& B{\l}ocki's conjecture; open for \(1<m<n\)\\
Finite-mass functions with relatively compact deep sublevel sets
& \(p<\pc\)
& Capacity estimate \cite{DinewKolodziej2014}\\
Cegrell class \(\mathcal E_{p,m}\)
& \(s<(m+p)n/(n-m)\)
& Finite \((p,m)\)-energy \cite{AhagCzyz2020}\\
Compact K{\"a}hler background
& \(p<n/(n-m)\)
& Arbitrary \((\omega,m)\)-subharmonic functions
  \cite{LuNguyen2015}\\
Compact Hermitian background
& \(p<n/(n-m)\)
& Global theorem \cite{Fang2026}\\
Local capacity condition \(\mathrm C_{m,\delta}\)
& \(s<(m+\delta)n/(n-m)\)
& Local capacity decay of order \(m+\delta\)\\
\bottomrule
\end{tabularx}
\end{table}

\section{B{\l}ocki's conjecture and the endpoint}
\label{sec:blocki}

The fundamental solution proves that the strong endpoint assertion
\(\SH_m\subset L^{\pc}_{\loc}\) is false.  The genuine conjecture is the
strict statement
\[
 \SH_m(\Omega)\subset L^p_{\loc}(\Omega)
 \quad\text{for every }p<\pc.
 \tag{6.1}\label{eq:blocki-conjecture}
\]
As summarized in \Cref{tab:results}, \eqref{eq:blocki-conjecture} has not
been established for arbitrary local \(m\)-subharmonic functions when
\(1<m<n\), and no counterexample is known in the cited literature.

A useful stronger target is the weak endpoint
\[
 u\in L^{\pc,\infty}_{\loc}.
\]
Indeed, if \(K\Subset\Omega\) and
\[
 \Vol(K\cap\{|u|>s\})\le C_K s^{-\pc}
 \qquad\text{for all sufficiently large }s,
\]
then the layer-cake formula gives \(u\in L^p_{\loc}\) for every
\(p<\pc\), as in \eqref{eq:blocki-conjecture}.  The fundamental solution
satisfies this estimate sharply, making the weak endpoint a natural
intermediate target.

\section{A scale of Hessian-capacity criteria}
\label{sec:capacity}

For a bounded \(m\)-hyperconvex domain \(D\) and a Borel set
\(E\Subset D\), define the relative Hessian capacity
\[
 \mCap_m(E,D)=
 \sup_{\substack{\psi\in\SH_m(D)\\-1\le\psi\le0}}
 \int_E(\ddc\psi)^m\wedge\beta^{n-m}.
\]

\begin{definition}
Let \(u\in\SH_m(\Omega)\) and \(\delta\ge0\).  We say that \(u\) satisfies
\(\Cmd\) if, for every
\(K\Subset\Omega\), there exist a bounded \(m\)-hyperconvex domain \(D\)
with \(K\Subset D\subseteq\Omega\), a constant \(A\in\mathbb R\), and
constants \(C,T>0\) such that
\[
 \mCap_m(K\cap\{u<A-t\},D)\le C t^{-(m+\delta)},
 \qquad t\ge T.
 \tag{\(\mathrm C_{m,\delta}\)}\label{eq:capacity-condition}
\]
The critical member \(\mathrm C_{m,0}\) is denoted by \(\mathrm C_m\).
\end{definition}

The volume-capacity estimate of Dinew--Ko{\l}odziej
\cite{DinewKolodziej2014} states that, for every
\[
 1<\tau<\frac n{n-m},
\]
there exists a constant \(C_{D,\tau}>0\) such that
\[
 \Vol(E)\le C_{D,\tau}\mCap_m(E,D)^\tau,
 \qquad E\Subset D.
 \tag{7.1}\label{eq:volume-capacity}
\]

\begin{theorem}
\label{thm:capacity}
If \(u\in\SH_m(\Omega)\) satisfies \(\Cmd\), then
\[
 u\in L^s_{\loc}(\Omega)
 \quad\text{for every}\quad
 0<s<\frac{(m+\delta)n}{n-m}.
\]
\end{theorem}

\begin{proof}
Fix \(K\Subset\Omega\) and choose the data in \(\Cmd\).  Given
\[
 s<\frac{(m+\delta)n}{n-m},
\]
choose
\[
 \max\left\{1,\frac{s}{m+\delta}\right\}
 <\tau<\frac n{n-m}.
\]
Combining \eqref{eq:volume-capacity} with \(\Cmd\) gives
\[
 \Vol(K\cap\{u<A-t\})
 \le C't^{-(m+\delta)\tau}.
\]
For \(X=(A-u)_+\), the layer-cake formula yields
\[
\begin{aligned}
 \int_K X^s\,dV
 &=s\int_0^\infty t^{s-1}\Vol(K\cap\{X>t\})\,dt\\
 &\le T^s\Vol(K)
   +sC'\int_T^\infty t^{s-1-(m+\delta)\tau}\,dt
 <\infty,
\end{aligned}
\]
because \((m+\delta)\tau>s\).  Since \(u\) is locally bounded above, its
positive part is bounded on \(K\), while \((-u)_+\le X+|A|\).  Hence
\(|u|^s\in L^1(K)\).
\end{proof}

The conditions \(\mathrm C_{m,\delta}\) are invariant under replacing \(u\)
by an \(m\)-subharmonic function that differs from \(u\) by a locally
bounded function.  They are also preserved under positive rescaling and
under taking the maximum with a locally bounded \(m\)-subharmonic function.
For a negative function \(v\in\mathcal E_{0,m}(D)\), the standard sublevel
comparison estimate is
\[
 t^m\mCap_m(\{v<-s-t\},D)
 \le\int_{\{v<-s\}}(\ddc v)^m\wedge\beta^{n-m},
 \qquad s,t>0.
 \tag{7.2}\label{eq:comparison-capacity}
\]
See \cite[Lemma 5.2 and its proof]{AhagCzyz2020}.  It follows that negative
\(m\)-subharmonic functions with finite total Hessian mass and relatively
compact deep sublevel sets satisfy \(\mathrm C_m\)
\cite{DinewKolodziej2014}.  For the
fundamental solution, the capacity of a ball scales as \(r^{2(n-m)}\),
whereas the sublevel radius scales as \(t^{-m/(2(n-m))}\); together these
give the decay \(t^{-m}\).

\begin{proposition}
\label{prop:energy-capacity}
Let \(D\) be a bounded \(m\)-hyperconvex domain and \(p\ge0\).  If
\(u\in\mathcal E_{p,m}(D)\), then \(u\) satisfies
\(\mathrm C_{m,p}\).  Consequently,
\[
 u\in L^s(D)
 \quad\text{for every}\quad
 0<s<\frac{(m+p)n}{n-m}.
 \tag{7.3}\label{eq:energy-sobolev}
\]
In particular, every \(u\in\mathcal E_m(D)\) satisfies \(\mathrm C_m\)
locally.
\end{proposition}

\begin{proof}
Write
\[
 e_{p,m}(u)=\int_D(-u)^p(\ddc u)^m\wedge\beta^{n-m}.
\]
The sublevel estimate of {\AA}hag--Czy{\.z}
\cite[Lemma 5.2]{AhagCzyz2020} gives
\[
 \mCap_m(\{u<-t\},D)
 \le 2^{m+p}e_{p,m}(u)t^{-(m+p)},\qquad t>0.
 \tag{7.4}\label{eq:energy-capacity}
\]
Thus \(u\) satisfies \(\mathrm C_{m,p}\) with \(A=0\).  Combining
\eqref{eq:energy-capacity} with the volume-capacity estimate and the
layer-cake formula, as in the proof of \Cref{thm:capacity}, gives
\eqref{eq:energy-sobolev}.  In the notation of \cite{AhagCzyz2020},
\(u\in\mathcal E_m(D)\) means that, for every \(\omega\Subset D\), the
restriction of \(u\) to \(\omega\) agrees with some
\(u_\omega\in\mathcal E_{0,m}(D)\).  Applying the case \(p=0\) to each
\(u_\omega\) proves the local assertion.
\end{proof}

\begin{remark}
For \(p>0\), the additional decay factor \(t^{-p}\) in
\eqref{eq:energy-capacity} is essential.  Using only the weaker critical
condition \(\mathrm C_m\) would yield \(s<nm/(n-m)\), rather than the
stronger {\AA}hag--Czy{\.z} range in \eqref{eq:energy-sobolev}.
\end{remark}

\begin{remark}
\Cref{thm:capacity} provides a sufficient criterion but does not prove the
unrestricted conjecture.  It would suffice to show that every local
\(m\)-subharmonic function satisfies \(\Cm\); a correspondingly sharp
weak-\(L^{\pc}\) volume-tail estimate would also settle
\eqref{eq:blocki-conjecture}.
\end{remark}

\section{Two problems of Benali--Ghiloufi}
\label{sec:bg-problems}

\subsection{Problem 2: the maximum formula}

For \(u\in\SH_m(\Omega)\), set
\[
 M(u,a,r)=\sup_{B(a,r)}u.
\]
Benali--Ghiloufi proved that \(M(u,a,r)\), viewed as a function of
\(\phi_m(r)\), is increasing and convex.  They also proved that
\[
 \ell_u(a):=2\lim_{r\downarrow0}
 \frac{M(u,a,r)}{\phi_m(r)}
\]
exists and satisfies \(\ell_u(a)\le\nu_u(a)\).  In our notation, their
Problem~2 is the following \cite[Problem~2]{BenaliGhiloufi2018}:

\begin{quote}
Is it true that, for every \(u\in\SH_m(\Omega)\) and every \(a\in\Omega\),
\[
 \ell_u(a)=\nu_u(a),
 \qquad\text{equivalently,}\qquad
 \nu_u(a)=2\lim_{r\downarrow0}
 \frac{M(u,a,r)}{\phi_m(r)}\,?
\]
\end{quote}

\begin{theorem}
\label{thm:bg-problem2}
Let \(1\le m<n\).  For every \(u\in\SH_m(\Omega)\) and \(a\in\Omega\),
\[
 \boxed{\quad
 \nu_u(a)=2\lim_{r\downarrow0}
 \frac{M(u,a,r)}{\phi_m(r)}.
 \quad}
 \tag{8.1}\label{eq:bg-maximum}
\]
\end{theorem}

\begin{proof}
By translating coordinates, we may assume that \(a=0\), and set
\[
 q=\frac nm-1=\frac{n-m}{m}>0.
\]
Consider the rescaled functions
\[
 u_r(w)=r^{2q}u(rw).
\]
The strong uniqueness theorem for tangents due to Harvey--Lawson, in the
form stated for \(m\)-subharmonic functions by Dinew--Ko{\l}odziej
\cite[Theorem 3.1]{DinewKolodziej2018}, gives
\[
 u_r\longrightarrow
 U(w)=-\frac{\nu_u(0)}{2q}|w|^{-2q}
 \quad\text{in }L^1_{\loc}(\C^n).
 \tag{8.2}\label{eq:tangent-convergence}
\]
The coefficient is determined by the Benali--Ghiloufi spherical-mean formula
\eqref{eq:lelong-spherical}, together with
\(\phi_m(r)=-(qr^{2q})^{-1}\).

Each \(u_r\) is subharmonic.  Hartogs' lemma, applied on a fixed ball
containing \(\overline{B(0,1)}\), therefore yields
\[
 \limsup_{r\downarrow0}r^{2q}M(u,0,r)
 =\limsup_{r\downarrow0}\sup_{B(0,1)}u_r
 \le\sup_{\overline{B(0,1)}}U
 =-\frac{\nu_u(0)}{2q}.
 \tag{8.3}\label{eq:hartogs-upper}
\]
On the other hand, the spherical mean satisfies
\(\lambda(u,0,r)\le M(u,0,r)\), while
\eqref{eq:lelong-spherical} gives
\[
 \lim_{r\downarrow0}r^{2q}\lambda(u,0,r)
 =-\frac{\nu_u(0)}{2q}.
 \tag{8.4}\label{eq:mean-limit}
\]
Together, \eqref{eq:hartogs-upper} and \eqref{eq:mean-limit} show that
\[
 r^{2q}M(u,0,r)\longrightarrow-\frac{\nu_u(0)}{2q}.
\]
Since \(\phi_m(r)=-(qr^{2q})^{-1}\), this is equivalent to
\eqref{eq:bg-maximum}.  The argument also covers the case \(\nu_u(0)=0\).
\end{proof}

\begin{remark}
For \(m=n\), one has \(q=0\), so the power rescaling above degenerates.  The
classical maximum formula for plurisubharmonic functions still holds,
although psh tangents need not be radial.  Thus the preceding proof relies
on a rigidity phenomenon specific to \(m<n\).
\end{remark}

\subsection{Problem 3: semicontinuity and polynomial openness}

For \(u\in L^1_{\loc}(\Omega)\) and \(a\in\Omega\), set
\[
 I_u(a):=\{p>0:|u|^p\text{ is integrable in a neighborhood of }a\}.
\]
In our notation, Benali--Ghiloufi's Problem~3 reads as follows
\cite[Problem~3]{BenaliGhiloufi2018}:

\begin{quote}
Let \(u\in\SH_m(\Omega)\).
\begin{enumerate}[label=\textup{(\arabic*)}]
\item Are the maps
\[
 a\longmapsto\iota_a(u)
 \qquad\text{and}\qquad
 u\longmapsto\iota_a(u)
\]
lower semicontinuous on \(\Omega\) and on the space of locally integrable
functions, respectively?
\item Is \(I_u(a)\) an open subset of \((0,\infty)\)?
\end{enumerate}
\end{quote}

For the function variable in part~\textup{(1)}, we interpret lower
semicontinuity with respect to the \(L^1_{\loc}\)-topology.  For a compact
set \(K\Subset\Omega\), we also write
\[
 I_u(K):=\{p>0:|u|^p\in L^1
 \text{ on a neighborhood of }K\}.
\]

\begin{theorem}
\label{thm:bg-problem3}
Let \(1\le m<n\).
\begin{enumerate}[label=\textup{(\roman*)}]
\item The map \(a\mapsto\iota_a(u)\) is lower semicontinuous on \(\Omega\)
for every locally integrable \(u\).
\item For fixed \(a\), the map \(u\mapsto\iota_a(u)\) is not lower
semicontinuous in the \(L^1_{\loc}\)-topology, even when restricted to
\(\SH_m(\Omega)\).
\item The interval \(I_u(K)\) need not be open, even when
\(u\in\SH_m(\Omega)\) and \(K\) is a singleton.
\end{enumerate}
\end{theorem}

\begin{proof}
Part~\textup{(i)} follows from \Cref{prop:point-lsc}.  For
part~\textup{(ii)}, choose
\[
 0<\alpha<c_m,\qquad v(z)=-|z|^{-2\alpha}
\]
on a small ball centered at zero.  By
\Cref{thm:radial-classification}, \(v\in\SH_m\), and
\Cref{prop:radial-lp} gives
\[
 \iota_0(v)=\frac n\alpha.
\]
Set \(v_j=v/j\).  Because \(\alpha<c_m\le n-1<n\), one has
\(v\in L^1_{\loc}\), and hence
\[
 v_j\longrightarrow0\quad\text{in }L^1_{\loc}.
\]
Multiplication by a positive constant does not affect polynomial
integrability, so
\[
 \iota_0(v_j)=\frac n\alpha
 \quad\text{for every }j,
 \qquad
 \iota_0(0)=+\infty.
\]
This violates lower semicontinuity at the zero function.

For part~\textup{(iii)}, take the function
\(\varphi_{\alpha,\gamma}\) in
\Cref{prop:closed-endpoint}, with
\(0<\alpha<c_m\) and \(\gamma>\alpha/n\).  Then
\[
 I_{\varphi_{\alpha,\gamma}}(\{0\})
 =\left(0,\frac n\alpha\right],
\]
which is not open.
\end{proof}

\begin{remark}
If a plurisubharmonic germ \(u\) at \(a\) is not identically \(-\infty\),
Skoda's exponential integrability theorem implies that
\(|u|^p\in L^1_{\loc}\) for every finite \(p\).  Thus
\(\iota_a(u)=+\infty\), and \(I_u(\{a\})=(0,\infty)\).  The failures in
\Cref{thm:bg-problem3}(ii)--(iii) are specific to \(m<n\) and should not be
confused with strong openness for exponential weights.
\end{remark}

\section{Radial functions and the Benali--Ghiloufi theorem}
\label{sec:radial-bg}

The mean-value theorem of Benali--Ghiloufi provides a direct way to verify
\(\Cm\) for radial functions.  This application is distinct from its role
in obtaining the upper bound in \eqref{eq:bg-exponent-bounds}.  Although the
maximum identity in \Cref{thm:bg-problem2} holds without radiality, it does
not by itself provide a pointwise lower bound.  For radial functions, the
mean asymptotic becomes pointwise.

\begin{theorem}
\label{thm:radial-capacity}
Let \(u\in\SH_m(B(a,R))\) be radial about \(a\), and suppose that
\(u\not\equiv-\infty\).  Then \(u\) satisfies \(\Cm\) locally.  Hence
\[
 u\in L^p_{\loc}(B(a,R))\quad\text{for all }p<\pc.
\]
More generally, the same local conclusion near \(a\) holds without the
radiality assumption whenever
\[
 u(z)\ge A\phi_m(|z-a|)-B\qquad(z\to a)
 \tag{9.1}\label{eq:radial-lower-bound}
\]
for some \(A,B>0\).
\end{theorem}

\begin{proof}
Set \(q=(n-m)/m\), so that
\(\phi_m(r)=-(qr^{2q})^{-1}\).  After shrinking the ball and adding a
constant, we may write \(u(z)=f(|z-a|)\le0\).  Since \(u\) is subharmonic,
\(f\) is increasing, and hence
\[
 M(u,a,r)=f(r).
\]
Thus \Cref{thm:bg-problem2} gives
\[
 \nu_u(a)=2\lim_{r\downarrow0}\frac{f(r)}{\phi_m(r)}.
 \tag{9.2}\label{eq:radial-lelong}
\]
For radial \(u\), this formula also follows directly from the
Benali--Ghiloufi spherical-mean identity \eqref{eq:lelong-spherical}, since
\(\lambda(u,a,r)=f(r)\).  The point of \Cref{thm:bg-problem2} is that no
radiality assumption is required.

To record the convexity underlying the next estimate, define
\[
 F(s)=f\bigl(\phi_m^{-1}(s)\bigr),\qquad s=\phi_m(r).
\]
Since \(M(u,a,r)=f(r)\), the Benali--Ghiloufi convexity theorem directly
implies that \(F\) is increasing and convex.  In the smooth case, the radial
Hessian test gives an alternative verification.  Write \(u(z)=g(t)\), where
\(t=|z-a|^2\), and set
\[
 A_0=g'(t),\qquad B_0=g'(t)+tg''(t).
\]
Because \(s=-t^{-q}/q\), one has \(ds/dt=t^{-q-1}\), and therefore
\[
 F'(s)=t^{q+1}A_0,\qquad
 F''(s)=t^{2q+1}(B_0+qA_0)\ge0.
\]
When \(A_0>0\), the final inequality is equivalent to
\(B_0/A_0\ge-q\) by \Cref{lem:radial-test}.  If \(A_0=0\), it follows from
\(\sigma_1=B_0\ge0\).  In the nonsmooth case, the conclusion is precisely
the Benali--Ghiloufi convexity theorem.  As \(r\downarrow0\), one has
\(s\to-\infty\).  Since \(F\) is convex and increasing, its recession slope
at \(-\infty\) is finite and bounded above by every finite interior chord
slope.  By \eqref{eq:radial-lelong}, this slope equals \(\nu_u(a)/2\).
Consequently, for some \(A>0\),
\[
 \frac{f(r)}{\phi_m(r)}\le A
\]
for small \(r\).  Since \(\phi_m<0\),
\[
 f(r)\ge A\phi_m(r),
\]
which gives \eqref{eq:radial-lower-bound} up to an additive constant.

Choose \(r_0>0\) so that \(u(z)\ge A\phi_m(|z-a|)\) on
\(B(a,r_0)\).  If \(z\) belongs to the sublevel set
\(\{u<-t\}\cap B(a,r_0)\), then
\[
 -\frac{A}{q|z-a|^{2q}}
 =A\phi_m(|z-a|)
 \le u(z)<-t.
\]
Multiplication by \(-1\) and rearrangement give
\[
 |z-a|<
 \left(\frac{A}{qt}\right)^{1/(2q)}.
\]
Hence, for all sufficiently large \(t\),
\[
 \{u<-t\}\cap B(a,r_0)
 \subset
 B\left(a,\left(\frac{A}{qt}\right)^{1/(2q)}\right).
\]
If the constant \(B\) in \eqref{eq:radial-lower-bound} is retained, the
same calculation replaces \(t\) by \(t-B\).  For \(t>2B\), this only changes
the constant.  After increasing the lower threshold for \(t\), the resulting
ball is contained in \(B(a,r_0)\).

Choose \(R_0\) such that \(r_0<R_0<R\).  The relative extremal function of
\(\overline{B(a,\rho)}\) in \(B(a,R_0)\) is
\[
 h_\rho(z)=
 \max\left\{-1,
 \frac{\phi_m(|z-a|)-\phi_m(R_0)}
      {\phi_m(R_0)-\phi_m(\rho)}\right\}.
\]
Computing its Hessian mass gives
\[
 \mCap_m(B(a,\rho),B(a,R_0))
 =c_{n,m}
   \bigl(\rho^{-2q}-R_0^{-2q}\bigr)^{-m}
 \le C_{n,m,R_0}\rho^{2qm}
 =C_{n,m,R_0}\rho^{2(n-m)},
\]
where \(c_{n,m}>0\) depends on the normalization of \(\ddc\).  This holds for
small \(\rho\).  With
\(\rho=(A/(qt))^{1/(2q)}\), monotonicity of relative capacity now yields
\[
 \mCap_m(\{u<-t\}\cap B(a,r_0),B(a,R_0))
 \le C\rho^{2(n-m)}
 =C't^{-(n-m)/q}
 =C't^{-m}.
\]
This verifies \(\Cm\) near \(a\).  On compact subsets away from the center,
the radial profile is bounded below, so sufficiently deep sublevel sets are
empty.  The same sublevel-set and capacity argument, beginning directly
with \eqref{eq:radial-lower-bound}, proves the nonradial assertion.
\end{proof}

The following corollary combines this theorem with
\eqref{eq:bg-exponent-bounds}.

\begin{corollary}
\label{cor:capacity-lelong}
If \(u\) satisfies \(\Cm\) in a neighborhood of \(x\), then
\(\iota_x(u)\ge\pc\).  If in addition \(\nu_u(x)>0\), then
\[
 \iota_x(u)=\pc.
\]
In particular, this equality holds for a radial \(m\)-subharmonic germ
with positive \(m\)-Lelong number.
\end{corollary}

When \(m=1\), the classical local weak-capacity estimate shows that ordinary
subharmonic functions satisfy \(\Cm\) locally, and the corollary recovers
\(\iota_x(u)=n/(n-1)\) at points with positive Lelong number.  For
\(1<m<n\), the positive-Lelong result of Benali--Ghiloufi provides only the
upper bound.  The matching lower bound follows once \(\Cm\) has been
verified, for example by the radial theorem.

\section{Future directions}
\label{sec:future}

\paragraph{The weak critical endpoint.}
A natural intermediate step toward B{\l}ocki's conjecture is to determine
whether every local \(m\)-subharmonic function belongs to
\(L^{\pc,\infty}_{\loc}\).  Equivalently, for every \(K\Subset\Omega\), one
seeks a constant \(C_K>0\) such that
\[
 \Vol(K\cap\{|u|>t\})\le C_K t^{-\pc},\qquad t>0.
\]
The exponent is sharp, as shown by the fundamental solution.  Such an
estimate would imply \(u\in L^p_{\loc}\) for every \(p<\pc\) by the
layer-cake formula.  Although \(\Cm\) yields this strict range, it does not
by itself establish the weak endpoint.

\paragraph{Intrinsic capacity criteria.}
It remains to characterize natural classes of \(m\)-subharmonic germs
satisfying \(\Cm\), and to determine whether the capacity scale
\(\mathrm C_{m,\delta}\) is also necessary for the corresponding
integrability ranges.  Borderline Lorentz and logarithmic refinements are
especially relevant because the power-logarithmic models show that the
leading power alone does not determine endpoint membership.

\paragraph{A replacement for Skoda theory.}
The counterexample in \Cref{thm:skoda-failure} shows that exponential
integrability cannot be controlled by the \(m\)-Lelong number alone.  One
should instead seek thresholds based on Hessian capacity, energy, or the
full tangent profile.  The capacity method applied to the finite-energy
range problem on compact manifolds in \cite{Lin2025}, together with the
local scale \(\mathrm C_{m,\delta}\) developed here, suggests compatible
starting points for such a theory.  Another question is whether stronger
energy topologies restore the semicontinuity of
\(u\mapsto\iota_a(u)\), which fails in the bare \(L^1_{\loc}\)-topology.

\textbf{Acknowledgments} The author thanks BIMSA for its hospitality.
ChatGPT (OpenAI) was used for language editing and exploratory discussion
during preparation of the manuscript. The author is solely responsible for
the mathematical content.

\begin{thebibliography}{99}
\raggedright

\bibitem{AhagCzyz2020}
P.~{\AA}hag and R.~Czy{\.z},
\emph{Poincar\'e- and Sobolev-type inequalities for complex
\(m\)-Hessian equations}, Results Math. \textbf{75} (2020), Article 63.

\bibitem{BenaliGhiloufi2018}
A.~Benali and N.~Ghiloufi,
\emph{Lelong numbers of \(m\)-subharmonic functions},
J. Math. Anal. Appl. \textbf{466} (2018), 1373--1392.

\bibitem{Berndtsson2013}
B.~Berndtsson,
\emph{The openness conjecture for plurisubharmonic functions},
arXiv:1305.5781, 2013.

\bibitem{Blocki2005}
Z.~B{\l}ocki,
\emph{Weak solutions to the complex Hessian equation},
Ann. Inst. Fourier (Grenoble) \textbf{55} (2005), 1735--1756.

\bibitem{DinewKolodziej2014}
S.~Dinew and S.~Ko{\l}odziej,
\emph{A priori estimates for complex Hessian equations},
Anal. PDE \textbf{7} (2014), 227--244.

\bibitem{DinewKolodziej2018}
S.~Dinew and S.~Ko{\l}odziej,
\emph{Non standard properties of \(m\)-subharmonic functions},
Dolomites Res. Notes Approx. \textbf{11} (2018), 35--50.

\bibitem{Fang2026}
Y.~Fang,
\emph{Integrability of \((\omega,m)\)-subharmonic functions on compact
Hermitian manifolds}, Ann. Polon. Math. (2026), online first,
doi:10.4064/ap250407-17-9.

\bibitem{GuanZhou2015}
Q.~Guan and X.~Zhou,
\emph{A proof of Demailly's strong openness conjecture},
Ann. of Math. (2) \textbf{182} (2015), 605--616.

\bibitem{HarveyLawson2018}
F.~R.~Harvey and H.~B.~Lawson, Jr.,
\emph{Tangents to subsolutions: existence and uniqueness, Part I},
Ann. Fac. Sci. Toulouse Math. \textbf{27} (2018), 777--848.

\bibitem{Lin2025}
G.~Lin,
\emph{On the finite energy range of \(m\)-Hessian operator},
Potential Anal. \textbf{63} (2025), 1335--1346.

\bibitem{LuNguyen2015}
C.~H.~Lu and V.-D.~Nguyen,
\emph{Degenerate complex Hessian equations on compact K{\"a}hler
manifolds}, Indiana Univ. Math. J. \textbf{64} (2015), 1721--1745.

\end{thebibliography}
\end{document}